\documentclass[12pt,reqno]{amsart}

\usepackage{color}
\usepackage{amsmath, amsthm, amssymb}
\usepackage{amsfonts}
\usepackage[ansinew]{inputenc}
\usepackage[dvips]{epsfig}
\usepackage{graphicx}
\usepackage[english]{babel}
\usepackage{hyperref}
\theoremstyle{plain}
\newtheorem{thm}{Theorem}[section]
\newtheorem{thm*}{Theorem}

\newtheorem{lem}[thm]{Lemma}
\newtheorem{prop}[thm]{Proposition}
\newtheorem{claim*}[thm*]{Claim}

\theoremstyle{definition}

\theoremstyle{remark}
\newtheorem{rem}[thm]{Remark}

\numberwithin{equation}{section}

\newcommand{\average}{{\mathchoice {\kern1ex\vcenter{\hrule height.4pt
width 6pt depth0pt} \kern-9.7pt} {\kern1ex\vcenter{\hrule
height.4pt width 4.3pt depth0pt} \kern-7pt} {} {} }}
\newcommand{\ave}{\average\int}
\def\R{\mathbb{R}}
\newcommand{\I}{{\rm I}}

\begin{document}

\title[Fully nonlinear parabolic equations with rough kernels]{Regularity for fully nonlinear nonlocal parabolic equations with rough kernels}

\author{Joaquim Serra}
\address{Universitat Polit\`ecnica de Catalunya, Departament de Matem\`{a}tica  Aplicada I, Diagonal 647, 08028 Barcelona, Spain}

\maketitle

\begin{abstract}
We prove space and time regularity for solutions of fully nonlinear parabolic integro-differential equations with rough kernels.
We consider parabolic equations $u_t = \I u$, where $\I$ is translation invariant
and elliptic with respect to the class $\mathcal L_0(\sigma)$ of Caffarelli and Silvestre,  $\sigma\in(0,2)$ being the order of $\I$.
We prove that if $u$ is a viscosity solution  in $B_1 \times (-1,0]$  which is merely bounded in $\R^n \times (-1,0]$, then $u$ is $C^\beta$
in space  and $C^{\beta/\sigma}$ in time in $\overline{B_{1/2}} \times [-1/2,0]$, for all  $\beta< \min\{\sigma, 1+\alpha\}$, where $\alpha>0$.
Our proof combines a Liouville type theorem ---relaying on the nonlocal parabolic $C^\alpha$ estimate of
Chang and D\'avila---  and a blow up and compactness argument.
\end{abstract}

\section{Introduction}
In \cite{CS}, Caffarelli and Silvestre introduced the ellipticity class $\mathcal L_0=\mathcal L_0(\sigma)$, with order $\sigma\in (0,2)$. The class $\mathcal L_0$ contains all linear operators $L$ of the form
\[ Lu(x) =  \int_{\R^n} \left(\frac{u(x+y)+u(x-y)}{2} -u(x)\right) K(y)\,dy,\]
where the kernels $K(y)$ satisfy the ellipticity bounds
\[0< \lambda\frac{2-\sigma}{|y|^{n+\sigma}} \le K(y)\le \Lambda\frac{2-\sigma}{|y|^{n+\sigma}}.\]
This includes kernels that may be very oscillating and irregular. That is why
the words {\em rough kernels} are sometimes used to refer to $\mathcal L_0$.
The extremal operators $M_\sigma^+$ and $M_\sigma^-$ for $\mathcal L_0$ are
\[  M_\sigma^+ u(x) = \sup_{L\in \mathcal L_0} Lu(x) \quad \mbox{ and }\quad  M_\sigma^- u(x) = \inf_{L\in \mathcal L_0} Lu(x).\]

If $u \in L^\infty(\R^n)$ satisfies the two viscosity inequalities $M_\sigma^+u\ge 0$ and $M_\sigma^-u \le 0$ in $B_1$, then $u$ belongs to $C^{\alpha}(\overline{B_{1/2}})$. More precisely, one has the estimate
\begin{equation}\label{Calphaelliptic}
\|u\|_{C^\alpha(B_{1/2})}\le C\|u\|_{L^\infty(\R^n)}.
\end{equation}
This estimate, with constants that remain bounded as the $\sigma\nearrow2$, is one of the main results in \cite{CS}.

For second order equations ($\sigma=2$) the analogous of \eqref{Calphaelliptic} is the classical estimate of Krylov and Safonov, and differs from \eqref{Calphaelliptic} only from the fact that it has
$\|u\|_{L^\infty(B_1)}$ instead of $\|u\|_{L^\infty(\R^n)}$ on the right hand side.
This apparently harmless difference comes from the fact that elliptic equations of order $\sigma<2$ are nonlocal.
By analogy with second order equations, from \eqref{Calphaelliptic} one expects to obtain $C^{1,\alpha}$ interior regularity of solutions to translation invariant elliptic equations $\I u =0$ in $B_1$.
When $\sigma=2$, this is done by applying iteratively the estimate \eqref{Calphaelliptic} to incremental quotients of $u$, improving at each step by $\alpha$
the H\"older exponent in a smaller ball (see \cite{CC}).
However, in the case $\sigma<2$ the same iteration does not work since, right after the first step, the $L^\infty$ norm of the incremental quotient of $u$ is only bounded in $B_{1/2}$, and not in the whole $\R^n$.

The previous difficulty is very related to
the fact that the operator will ``see'' possible distant high frequency oscillations in the exterior Dirichlet datum.
In \cite{CS}, this issue is bypassed by restricting the ellipticity class, i.e., introducing a new class $\mathcal L_1\subset \mathcal L_0$ of operators with $C^1$
kernels (away from the origin). The additional regularity of the kernels has the effect of averaging distant high frequency oscillations, balancing out its influence.
This is done with an integration by parts argument.
Hence, the  $C^{1+\alpha}$ estimates in \cite{CS} are ``only'' proved for elliptic equations with respect to $\mathcal L_1$ (instead of $\mathcal L_0$).

Very recently, Dennis Kriventsov \cite{K} succeeded in proving the same $C^{1+\alpha}$ estimates for elliptic equations of order $\sigma>1$ with rough kernels, that is, for $\mathcal L_0$.
The proof in \cite{K} is quite involved and combines fine new estimates with a compactness argument.
In \cite{K} the same methods are used to obtain other interesting applications, including nearly sharp Schauder type estimates for linear, non translation invariant, nonlocal elliptic equations.

Here, we extend the main result in \cite{K} in two ways, providing in addition a new proof of it. First, we pass from elliptic to parabolic equations. Second, we allow also $\sigma\le1$, proving in this case 
$C^{\sigma-\epsilon}$ regularity in space and $C^{1-\epsilon}$ in time (for all $\epsilon>0$) for solutions to nonlocal translation invariant parabolic equations with rough kernels. Our proof follows a new method, different from that in \cite{K}. As explained later in this introduction, our strategy is to prove first a Liouville type theorem for global solutions,  and  to deduce later the interior estimates from this Liouville theorem, using a blow up and compactness argument.
That a regularity estimate and a Liouville theorem are in some way equivalent is an old principle in PDEs, but here it turns out to be very useful to bypass the difficulty iterating the ``nonlocal'' estimate \eqref{Calphaelliptic}.

Therefore, a main interest of this paper lies precisely on the method that we introduce here. It is very flexible and can be useful in different contexts with nonlocal equations. For instance, the method can
be used to study equations which are nonlocal also in time, and also to analyze boundary regularity for nonlocal equations (see Remark \ref{appl}).

To have a local $C^{1+\alpha}$ estimate for solutions that are merely bounded in $\R^n$, it is necessary that the order of the equation be greater than one.
Indeed, for nonlocal equations of order $\sigma$ with rough kernels there is no hope to prove a local H\"older estimate of order greater than $\sigma$ for solutions that are merely bounded in $\R^n$.
The reason being that influence of the distant oscillations is too strong. Counterexamples can be easily constructed even for linear equations. That is why the condition $\sigma>1$ is necessary for the $C^{1,\alpha}$ estimates of Kriventsov \cite{K}. Also, this is why we prove $C^\beta$ estimates in space only for $\beta<\sigma$.

As explained above, the difficulty of nonlocal equations with rough kernels, with respect to local ones, is that the estimate  \eqref{Calphaelliptic} is not immediately useful to prove higher order H\"older regularity
for solutions of $\I u =0$ in $B_1$. Recall that the classical iteration fails because, after the first step, the $L^\infty$ norm of the incremental quotient of order $\alpha$ is only controlled in $B_{1/2}$, and not in the whole $\R^n$.
The idea in our approach is that the iteration does work if one considers a solution in the whole space.
If we have a global solution $u$, then we can apply  \eqref{Calphaelliptic}  at every scale and deduce that $u$ is $C^\alpha$ in all space. Then, we consider the incremental quotients of order $\alpha$ of $u$, which we control in the whole $\R^n$,  and we prove that $u$ is $C^{2\alpha}$.  And so on. When this is done with estimates, taking into account the growth at infinity of the function $u$ and the scaling of the estimates, we obtains a Liouville theorem.
Using it, we deduce the higher order interior regularity of $u$ directly, using a blow up argument and compactness argument.
In order to have  compactness of sequences of viscosity solutions we only need the  $C^\alpha$ estimate  \eqref{Calphaelliptic}.

For local translation invariant elliptic equations like $F(D^2u)=0$  in $B_1$ it would be a unnecessary complication to first prove the Liouville theorem
and then obtain the interior estimate by the blow up and compactness argument in this paper. Indeed,  as said above, the iteration already works in the bounded domain $B_1$.
Nevertheless, it is worth noting that equations of the type $F(D^2u,Du,x)=0$, with continuous dependence on $x$,  become $\tilde F(D^2u)=0$ after blow up at some point.
By this reason, one can see that the second order Liouville theorem and the blow up method  provide a $C^{1,\alpha}$
bound for solutions to $F(D^2u,Du,x)=0$ in $B_1$. However, this approach gives nothing new with respect to classical perturbative methods (as in \cite{CC}).

For nonlocal equations, we could also have considered non translation invariant equations ---with continuous dependence on $x$---,
and having also lower order terms. This is because in our argument we blow up the equation. Translation and scale invariances are only needed in the limit equation (after blow up), to which we apply the Liouville theorem.
And, in a typical situation, when one blows up a non translation invariant equations with lower order terms one gets a translation invariant equation with no lower order terms.
Hence, in the appropriate setting, we could certainly extend our results to these equations.
In this paper, however, we do not include this since we are not interested in pushing the method to its limits, but rather in giving a clear example of its use.

In the following remark we give two examples of different contexts in which the method of this paper is useful.
\begin{rem}\label{appl}
{\em Nonlocal dependence also on time.}
Let us  point out that it is not essential  to our argument that that the equation is local in time. Hence, the same ideas  could be useful when considering nonlinear parabolic-like equations
which have a nonlocal dependence on the past time. For instance,  it could be useful when studying the nonlinear versions of the generalized master equations \cite{CSMasters}.

{\em Boundary regularity.} A boundary version of the method in the present paper turns out to be a powerful tool in the study of the boundary regularity for fully nonlinear integro-differential elliptic equations;
this is done in the work of Ros-Oton and the author \cite{RS}. In this case, the Liouville theorem to be used is for solutions in a half space $\{x_n>0\}$,
which clearly corresponds to the blow up of a smooth domain at a given boundary point. Interestingly, the possible solutions in this Liouville theorem are not planes,
but instead they are of the type $c (x_n)_+^s$, for some constant $c$.
Once one has this ``boundary'' Liouville theorem ---its proof is more involved than that of the ``interior'' one in this paper---, then the blow up and compactness argument in this paper can be adapted  to obtain fine boundary regularity results.
\end{rem}

\section{Main result}

The basic parabolic $C^{\alpha}$ estimate on which all our argument relies has been obtained by Chang and D\'avila \cite{ChD} ---this is the parabolic version of \eqref{Calphaelliptic} and we state it below.

In order that the statements of the results naturally include their classical second order versions, it is convenient to define the ellipticity class $\mathcal L_0(2)$, as the set of second order linear operators
\[Lu(x) = a_{ij}\partial_{ij} u(x)\]
with $(a_{ij})$ satisfying
\[ 0 < c_n \lambda {\rm Id} \le (a_{ij}) \le c_n \Lambda  {\rm Id} ,.\]
The constant $c_n$ is a appropriately chosen so that the operators in $\mathcal L_0(\sigma)$ converge to operators in $\mathcal L_0(2)$ (when applied to bounded smooth functions).

Throughout the paper, $\omega_{\sigma_0}$ denotes the weight
\[ \omega_{\sigma_0}(x) = \frac{2-\sigma_0}{1+|x|^{n+\sigma_0}}.\]

\begin{thm}[Regularity in space from {\cite[Theorem 5.1]{ChD}}]\label{holder}
Let $\sigma_0\in(0,2]$ and $\sigma\in[\sigma_0,2]$.
Let $u\in C\bigl(\overline{B_1} \times[-1,0]\bigr)$ with $\sup_{t\in[-1,0]}  \int_{\R^n} u(x,t) \omega_{\sigma_0}(x)\,dx <\infty$  satisfy the following two inequalities in the viscosity sense
\[ u_t -M_\sigma^+ u \le C_0 \quad \mbox{and}\quad u_t - M_\sigma^- u \ge -C_0\quad \mbox{in } B_1 \times(-1,0].\]
Then, for some $\alpha\in(0,1)$ and $C>0$, depending only on $\sigma_0$, ellipticity constants, and dimension, we have
\[\sup_{t\in[-1/2,0]} \bigl[u(\,\cdot\,,t)\bigr]_{C^\alpha(B_{1/2})} \le C\left( \|u\|_{L^\infty(B_1\times(-1,0])} + \sup_{t\in [-1,0]} \| u(\,\cdot\,,t)\|_{L^1(\R^n, \omega_{\sigma_0})} +C_0\right).\]
\end{thm}

Theorem 5.1 of \cite{ChD} contains also a $C^{\alpha/\sigma}$ estimate in time, for some $\alpha>0$. However,
for our argument we only need the estimate in space from \cite{ChD}, which is the one stated above.

Before stating our main result let us briefly recall some definitions (translation invariant elliptic operator, viscosity solution, etc.), which are by now standard
in the context of integro-differential equations.
They can be found in detail in \cite{CS, ChD}.

As in \cite{CS}, an operator $\I$ is said to be elliptic with respect to $\mathcal L_0(\sigma)$, $\sigma\in[\sigma_0,2]$, if
\[ M_\sigma^- (u-v)(x) \le  \I u(x)-\I v(x) \le M_\sigma^+(u-v)(x),\]
for all elliptic test functions $u,v$ at $x$, which are $C^2$ functions in a neighborhood of $x$ and having finite integral against the weight $\omega_{\sigma_0}$.
Recall that $\I$ is defined as a ``black box'' acting on test functions, with the only assumption that if $u$ is a test function at $x$, then $\I u$ is continuous near $x$.
The operator we have in mind is
\[ \I u(x) = \inf_{\alpha}\sup_\beta \bigl( L_{\alpha\beta} u + c_{\alpha,\beta}\bigl)\]
where $L_{\alpha\beta}\in \mathcal L_0(\sigma)$ and $\inf_{\alpha}\sup_\beta  c_{\alpha,\beta}=0$.

That $\I$ is translation invariant clearly means
\[\I \bigl(u(x_0 +\cdot )\bigr)(x) = (\I u) (x_0+x).\]

The definition we use of viscosity solutions (and inequalities) for parabolic equations  is the one in \cite{ChD}.
Namely, let $f$  and $u$ such be continuous functions in  a parabolic domain. Assume that
$\int_{\R^n} u(x,t) \omega_{\sigma_0}(x)\,dx$ is locally bounded  for all  $t$  in the domain. Then,  we say that   $u$ is a viscosity solution of
\[u_t -\I u =f\]
 if whenever a test function $v(x,t)$ touches by above (below) $u$ at $(x_0,t_0)$
we have $\bigl(v_{t^-} - \I v\bigr)(x_0,t_0) \le f(x_0,t_0)$  ($\ge$). For parabolic equations $v$ touching $u$  by above at $(x_0,t_0)$ means $v(x,t)\ge u(x,t)$ for all $x\in \R^n$ and for all $t\le t_0$.
As in \cite{ChD},  test functions $v$  are quadratic functions in some small cylinder and outside they have the same type of growth
as the solutions $u$. That is,
\[ v(x,t) =
a_{ij} x_i x_j + b_i x_i + ct +d \quad \mbox{ in the cylinder } \overline{B_\epsilon(x_0)} \times[t_0-\epsilon, t_0]\\
\]
for some $\epsilon>0$ and
\[ \|v(\,\cdot \,, t)\|_{L^1(\R^n,\omega_{\sigma_0})} = \int_{\R^n} |v(x,t)| \omega_{\sigma_0}(x)\,dx\]
if locally bounded for $t$ in the domain of the equation. As explained in \cite{ChD}, in order to have left continuity in time of $(\partial_t -\I) v(x,t)$,
one additionally requires test functions to satisfy $\lim_{t\nearrow t_0} \|v(\,\cdot \,, t)-v(\,\cdot \,, t_0)\|_{L^1(\R^n,\omega_{\sigma_0})} = 0$ for all $t_0$ in the domain.

Our main result is the following.
\begin{thm}\label{thm1}
Let $\sigma_0\in(0,2)$ and $\sigma\in{[\sigma_0,2]}$.
Let $u\in L^{\infty}\bigl(\R^n \times {(-1,0)}\bigr)$ be a viscosity solution of $u_t - \I u =f$ in $B_1\times(-1,0]$, where $\I$ is a translation
invariant elliptic operator with respect to the class $\mathcal L_0(\sigma)$ with $\I0=0$.
Let $\alpha=\alpha(\sigma_0)$ be given by Theorem \ref{holder}.

Then, for all $\epsilon>0$,  letting
\[\beta= \min\{\sigma, 1+\alpha\}-\epsilon,\]
 $u(\,\cdot\,,t)$
belongs to $C^\beta \left(\overline {B_{1/2}}\right)$ for all $t\in[-1/2,0]$, and $u(x,\,\cdot\,)$ belongs to $C^{\beta/\sigma}\left([-1/2,0]\right)$ for all $x\in B_{1/2}$. Moreover, the
following estimate holds
\[
\sup_{t\in[-1/2,0]} \bigl\|u(\,\cdot\,,t)\bigr\|_{C^\beta(B_{1/2})}+ \sup_{x\in B_{1/2}} \bigl\|u(x,\,\cdot\,)\bigr\|_{C^{\beta/\sigma}\left([-1/2,0]\right)} \le
 C C_0,
\]
where
\[C_0 =   \bigl\|u\|_{L^\infty\left(\R^n\times(-1,0)\right)} + \|f\|_{L^\infty(B_1\times(-1,0))}\]
and $C$ is a constant depending only on $\sigma_0$, $\epsilon$, ellipticity constants, and dimension.
\end{thm}

\section{Liouville type theorem}
As said in the introduction, Theorem \ref{thm1} will follow from a Liouville type theorem, which we state below, and a blow up and compactness augment.

In all the paper, given $\sigma\in (0, 2]$ and $R>0$, $Q_R^\sigma$ denotes the parabolic cylinder
\begin{equation}\label{cylinders}
Q_R^\sigma := \bigl\{ (x,t)\,: \, |x|\le R \mbox{ and } -R^\sigma < t < 0 \bigr\}.
\end{equation}
For $z\in \R^{n}\times(\infty,0]$, the cylinder $z + Q_R^\sigma$  is  denoted as $Q_R^\sigma(z)$.

\begin{thm}\label{liouville}
Let $\sigma_0\in(0,2)$, $\sigma\in [\sigma_0,2]$,  and $\alpha=\alpha(\sigma_0)$ be given by Theorem \ref{holder}.
Let  $0<\beta<\min\{\sigma_0, 1+\alpha\}$. Let $\I$ be a translation invariant operator, elliptic with respect to $\mathcal L_0(\sigma)$, with $\I0=0$.
 Assume that $u$ in $C\bigl(\R^n \times(-\infty,0]\bigr)$ satisfies the growth control
\begin{equation}\label{growthcontrol}
\|u\|_{L^\infty(Q_R^{\sigma})}\le CR^\beta\quad \mbox{for all }R\ge 1
\end{equation}
and that it is a viscosity solution of
\[u_t = \I u \quad \mbox{in all of }\,\R^n\times(-\infty,0].\]

Then,  if $\beta<1$, $u$ is constant. And if $\beta\ge1$,  $u(x,t)=a\cdot x+b$ is an affine function of the $x$ variables only.
\end{thm}

\begin{proof}
For all $\rho\ge1$ consider $v_\rho(x,t) = \rho^{-\beta} u(\rho x,\rho^\sigma t)$. Note that the growth control on
$u$ is transferred to $v_\rho$. Indeed,
\[ \|v_\rho\| _{L^\infty(Q_R^{\sigma})} = \rho^{-\beta}\|u\| _{L^\infty(Q_{\rho R}^\sigma)} \le C\rho^{-\beta}  (\rho R)^{\beta} = CR^\beta \quad \mbox{for all }R\ge 1\]
Hence, since $\beta<\sigma_0$,
\[ \sup_{t\in [-1,0]} \| v_\rho(\,\cdot\,,t)\|_{L^1(\R^n, \omega_{\sigma_0})} \le C(n, \sigma, \beta).\]
Moreover, since $u$ is satisfies $u_t \le M_\sigma^+ u$ and $u_t \ge M_\sigma^- u$ in $\R^n\times(-\infty,0]$, also $v_\rho$ satisfies the same inequalities.

By applying Theorem \ref{holder} to the function $v_\rho$ we obtain
\[\sup_{t\in[- 1/2,0]} \bigl[v_\rho(\,\cdot\,,t)\bigr]_{C^\alpha(B_{1/2})} \le C.\]
Scaling back to $u$ the previous estimate (setting $\rho= 2^{1/\sigma}R$) we obtain
\[\sup_{t\in(- R^\sigma,0]} \bigl[u(\,\cdot\,,t)\bigr]_{C^\alpha(B_{R})} \le C R^{\beta-\alpha}.\]
In this way for all $h\in \R^n$ we have an improved growth control for the incremental quotient
\[ v^{(\alpha)}_h (x,t) = \frac{u(x+h, t)-u(x,t)}{|h|^{\alpha}}.\]
Namely,
\[
\|v^{(\alpha)}_h\|_{L^\infty(Q_R^\sigma)}\le CR^{\beta-\alpha}\quad \mbox{for all }R\ge 1.
\]
Now, $v^{(\alpha)}_h$ satisfies again $\bigl(v^{(\alpha)}_h\bigr)_t \le M_\sigma^+ v^{(\alpha)}_h$ and $\bigl(v^{(\alpha)}_h\bigr)_t  \ge M_\sigma^- v^{(\alpha)}_h$. Hence, we may repeat
the previous scaling augment and use Theorem \ref{holder} to obtain
\[\sup_{t\in[- R^\sigma,0]} \bigl[ v^{(\alpha)}_h(\,\cdot\,,t)\bigr]_{C^\alpha(B_{R})} \le C R^{\beta-2\alpha}.\]
And this provides an improved growth control for
\[ v^{(2\alpha)}_h (x,t) = \frac{u(x+h, t)-u(x,t)}{|h|^{2\alpha}},\]
that is,
\[
\|v^{(2\alpha)}_h\|_{L^\infty(Q_R^\sigma)}\le R^{\beta-2\alpha}\quad \mbox{for all }R\ge 1.
\]
It is clear that we may keep iterating in this way, improving the growth control by $\alpha$ at each step.

After a bounded number of $N$ of iterations we  will have $(N+1)\alpha>1$ and we will obtain
\[\sup_{t\in[- R^\sigma,0]} \bigl[ v^{(N\alpha)}_h(\,\cdot\,,t)\bigr]_{C^\alpha(B_{R})} \le C R^{\beta-N\alpha},\]
which implies
\begin{equation}\label{estimatelast1}
\|v^{(1)}_h\|_{L^\infty(Q_R^\sigma)}\le R^{\beta-1}\quad \mbox{for all }R\ge 1.
\end{equation}
As usual with fully nonlinear equations we may do a last iteration to obtain a $C^{1,\alpha}$ estimate by using that $v^{(1)}_h$ satisfies the
two viscosity inequalities. Thus, using one more time Theorem \ref{holder} at every scale and \eqref{estimatelast1}
we obtain
\[
\left\| \frac{D_x u(x+h,t) - D_x u(x,t) }{|h|^\alpha} \right\|_{L^\infty(Q_R)}\le R^{\beta-1-\alpha}\quad \mbox{for all }R\ge 1.
\]
Above $D_x$ denotes any derivative with respect to some of the space variables.

Therefore, since by assumption $\beta<1+\alpha$,  sending $R\nearrow\infty$ we obtain
\[D_x u(x+h,t)=D_xu(x,t)\quad \mbox{for all }h\in\R^n.\]
Hence, $D_x u$ depends on the variable $t$ only. But since $D_x u$ satisfies
\[(D_xu)_t \le M_\sigma^+ (D_x u)=0\quad \mbox{and}\quad (D_xu)_t \ge M_\sigma^- (D_x u)=0\]
then it is $(D_xu)_t = 0$ in all of $\R^n\times (-\infty,0]$. Therefore,
\[u(x,t) = a\cdot x + \psi(t).\]

Finally, since $u_t= \I u=0$ we have $\psi(t)=b$ for some constant $b\in \R$.  Moreover, in the case $\beta<1$
the growth control yields $a=0$.
\end{proof}

\section{Preliminary lemmas and proof of Theorem \ref{thm1}}
As said above the proof of Theorem \ref{thm1} is by compactness. The following result is a consequence of the theory in \cite{CS2}
and provides compactness under weak convergence of sequences of elliptic operators which are elliptic with respect to some $\mathcal L_0(\sigma)$, with $\sigma$
varying in the interval $[\sigma_0,2]$.
We use the definition from \cite{CS2} of weak convergence of operators.
Namely, a sequence of translation invariant operators $\I_m$ is said to converge weakly to $\I$  if
for all $\epsilon>0$ and test function $v$, which is a quadratic polynomial in  $B_{\epsilon}$ and belongs to $L^1(\R^n,\omega_{\sigma_0})$, we have
\[  \I_m v (x) \rightarrow \I v(x)\quad \mbox{uniformly in }\overline{B_{\epsilon/2}}.\]
\begin{lem}\label{lemweak}
Let $\sigma_0\in (0,2)$, $\sigma_m\in [\sigma_0,2]$, and $\I_m$ such that
\begin{itemize}
\item $\I_m$ is translation invariant and elliptic with respect to $\mathcal L_0(\sigma_m)$.
\item $\I_m 0=0$.
\end{itemize}
Then, a subsequence of $\sigma_m\to \sigma \in [\sigma_0,2]$ and a subsequence of $\I_m$ converges weakly to some translation
invariant operator $\I$ elliptic with respect to $\mathcal L_0(\sigma)$.
\end{lem}
\begin{proof}
We may assume by taking a subsequence that $\sigma_m\to \sigma \in [\sigma_0,2]$.
Consider the class $\mathcal L = \bigcup_{\sigma\in[\sigma_0,2]} \mathcal L_0(\sigma)$. This class satisfies Assuptions 23 and 24
of \cite{CS2}. Also,  each $\I_m$ is elliptic with respect to $\mathcal L$. Hence using Theorem 42 in \cite{CS2} there is a subsequence
of $\I_m$ converging weakly (with respect to the weight $\omega_{\sigma_0}$) to a translation invariant operator $\I$, also elliptic with respect to $\mathcal L$. To see that $\I$ is in fact elliptic with
respect to $\mathcal L_0(\sigma)\subset \mathcal L$ we just observe that for test functions $u$ and $v$ that are quadratic polynomials in a neighborhood of $x$ and that belong to  $L^1(\R^n,\omega_{\sigma_0})$,
the inequalities
\[ M^-_{\sigma_m} v(x)\le \I_m(u+v)(x)-\I_m u(x)\le M^+_{\sigma_m} v(x) \]
pass to the limit to obtain
\[ M_{\sigma}^-v(x)\le \I(u+v)(x)-\I u(x)\le M_\sigma^+ v(x).\]
\end{proof}

The following result from \cite{ChD2} is a parabolic version of Lemma 5 in \cite{CS2}. It is the basic stability result which is needed in compactness arguments.
\begin{lem}[Reduced version of {\cite[Theorem 5.3]{ChD2}}]\label{lem5parab}
Let $R>0$, $T>0$, and $\I_m$ be a sequence of translation invariant elliptic operators. Let
$u_m \in C\bigl(\overline{B_R}\times[-T,0]\bigr)$  be viscosity solutions of
\[\partial_ t u_m - \I_m u_m =f_m\quad \mbox{in }B_R\times(-T,0].\]
Assume that
\[\I_m \rightarrow \I\quad \mbox{weakly with respect to }\omega_{\sigma_0}, \]
\[u_m(x,t)\rightarrow u(x,t) \quad \mbox{uniformly in } \overline{B_R}\times[-T,0],\]
\[  \sup_{t\in[-T,0]}\int_{\R^n} \bigl|u_m-u\bigr|(x,t) \omega_{\sigma_0}(x)\,dx \longrightarrow 0,\]
and
\[f_m \rightarrow 0  \quad \mbox{uniformly in } \overline{B_R}\times[-T,0].\]

Then, $u$ is a viscosity solution of $\partial_ t u =\I u$ in  $B_R\times(-T,0]$.
\end{lem}

The following useful lemma states that if in a sequence of nested sets a function $u$ is close enough to  its ``least squares''  fitting plane,  then $u$ is $C^\beta$ with $\beta \in(1,2)$.
\begin{lem}\label{lemaux}
Let $\sigma\in(1,2]$, $\beta\in(1,\sigma)$, and let $u$ be a continuous function belonging to $L^\infty(Q_{\infty})$, where $Q_\infty =\R^n\times (-\infty,0]$.
For $z = (z',z_{n+1})\in \R^n\times (-\infty,0]$ and $r>0$, define the constant in $t$ affine function
\begin{equation}\label{ell}
\ell_{r,z}  (x,t) := a^*\cdot (x-z') +b^*,
\end{equation}
where
\begin{equation}
 a^*_i = a^*_i (r,z)  = \frac{\int_{Q^\sigma_r(z)} u(x,t)(x_i -z_i) \,dx\,dt}{\int_{Q^\sigma_r(z)} (x_i-z_i)^2\,dx\,dt},\quad1\le i\le n,
\end{equation}
and
\begin{equation}
 b^* = b^* (r,z)  = \ave_{Q^\sigma_r(z)} u(x,t) \,dx\,dt,
\end{equation}
where $Q^r_\sigma(z)$ was defined in \eqref{cylinders}
Equivalently,
\[(a^*,b^*)= {\rm arg\,min} \int_{Q^\sigma_r(z)} \bigl(u(x,t) - a\cdot (x-z') +b \bigr)^2\,dx\,dt.\]

If for some constant $C_0$ we have
\begin{equation}\label{allrz}
\sup_{r>0} \sup_{z\in Q_{1/2}} r^{-\beta}\bigr\| u - \ell_{r,z} \bigr\|_{L^\infty(Q^\sigma_r(z))} \le C_0,
\end{equation}
where $Q_{1/2}= B_{1/2}\times (-1/2,0]$, then
\begin{equation}\label{holderbound}
\sup_{t\in[-1/2,0]} \bigl\|u(\,\cdot\,,t)\bigr\|_{C^\beta(B_{1/2})}+ \sup_{x\in B_{1/2}} \bigl\|u(x,\,\cdot\,)\bigr\|_{C^{\beta/\sigma}\left([-1/2,0]\right)} \le C\bigl(\|u\|_{L^\infty(Q_\infty)} + C_0\bigr),
\end{equation}
where $C$ depends only on $\beta$.
\end{lem}

\begin{proof}
We may assume $\|u\|_{L^\infty(Q_\infty)}=1$.
Recall the definition of $Q_r^\sigma$ in \eqref{cylinders}.
Note that \eqref{allrz} implies that for all $z\in \overline{Q_{1/2}}$, $r>0$, and $\bar z  \in  Q_r^\sigma(z)$ we have
\[\begin{split}
\bigl| \ell_{2r,z} (\bar z) - \ell_{r,z} (\bar z)\bigr|
&\le  \bigl| u(\bar z) -\ell_{2r,z} (\bar z)\bigr| + \bigl| u(\bar z) -\ell_{r,z} (\bar z)\bigr|\\
&\le C C_0 r^{\beta}.
\end{split}\]
But this happening for every $\bar z \in  Q_r^\sigma(z)$ means
\[ \bigl| a^*(2r,z) - a^*(r,z) \bigr| \le C C_0 r^{\beta-1}\]
and
\[ \bigl| b^*(2r,z) - b^*(r,z) \bigr| \le C C_0 r^{\beta}.\]
In addition since $\|u\|_{L^\infty(Q_\infty)}= 1$ we clearly have that
\begin{equation}\label{boundab}
| a^*(1,z)|\le C \quad \mbox{and}\quad |b^*(r,z)|\le 1\quad  \mbox{for all }r>0.
\end{equation}

Since $\beta>1$ this implies the existence of the limits
\[ a(z) := \lim_{r\searrow 0} a^*(r,z) \quad \mbox{and}\quad b(z):= \lim_{r\searrow 0} b^*(r,z).\]
Moreover,
\[
\begin{split}
\bigl| a(z) - a^*(r,z) \bigr| &\le  \sum_{m=0}^\infty \bigl| a^*(2^{-m}r,z) - a^*(2^{-m-1}r,z) \bigr|\\
&\le  \sum_{m=0}^\infty C C_0 2^{-(\beta-1)m}  r^{\beta-1} \le C(\beta) C_0 r^{\beta-1}.
\end{split}
\]
And similarly
\[
\bigl| b(z) - b^*(r,z) \bigr| \le  C(\beta) C_0 r^{\beta}.
\]
Using  \eqref{boundab} we obtain
\begin{equation}\label{boundab2}
|a(z)|\le C(\beta) (C_0+1)  \quad \mbox{and}\quad |b(z)|\le 1.
\end{equation}

We have thus proven that for all $z\in \overline{Q_{1/2}}$ there are $a(z)\in \R^n$ and $b(z)\in \R$ satisfying the bounds \eqref{boundab2}
such that for all $r>0$
\[
\bigl\| u - a(z)\cdot x -b(z)\bigr\|_{L^\infty(Q_r^\sigma(z))} \le C(\beta) C_0 r^{\beta}
\]
This implies that $u$ is differentiable in the $x$ directions, that $a(z)= D_x u(z)$ and $b(z)= u(z)$,
and that \eqref{holderbound} holds.
\end{proof}

The following standard lemma will be used to show that rescaled functions in the blow up argument also satisfy elliptic equations with the same ellipticity constants.
\begin{lem}\label{lemscalings}
Let $\sigma>1$ and $\I$ be a translation invariant operator with respect to $\mathcal L_0(\sigma)$ with $\I 0=0$. Given $x_0\in \R^n$, $r>0$, $c >0$, and $\ell(x)=a\cdot x+ b$, define
${\rm \tilde I}$ by
\[  {\rm \tilde I} \left(\frac{w(x_0+r\,\cdot)-\ell(x_0+r\,\cdot)}{c}\right) = \frac{r^\sigma}{c}(\I w)(x_0+r\,\cdot\,). \]
Then ${\rm \tilde I}$ is also translation invariant and elliptic with respect to $\mathcal L_0(\sigma)$ (with the same ellipticity constants) with ${\rm \tilde I}\,0=0$.
\end{lem}
\begin{proof}
We have
\[
\begin{split}
{\rm \tilde I}\, u(x) &=  \frac{r^\sigma}{c} \I\left( c u\left(\frac{\,\cdot\,-x_0}{r}\right)+\ell(\cdot)\right)(x_0 + rx)\\
&=  \frac{r^\sigma}{c} \I\left( c u\left(\frac{\,\cdot\,-x_0}{r}\right)\right)(x_0 + rx),
\end{split}
\]
where we have used $M_\sigma^+ \ell = M_\sigma^- \ell = 0$.

We clearly see from the second expression that $\tilde \I$ is translation invariant.

Also,
\[ {\rm \tilde I}\, 0 =  \frac{r^\sigma}{c} \I 0  = 0.\]

Moreover,
\[ \begin{split}
\bigl\{{\rm \tilde I}\, u-{\rm \tilde I}\, v\bigr\}(x) &=  \frac{r^\sigma}{c}\left\{ \I\left( c u\left(\frac{\,\cdot\,-x_0}{r}\right)+\ell\right)-\I\left( c v\left(\frac{\,\cdot\,-x_0}{r}\right)+\ell\right)\right\}(x_0 + rx)\\
&\le \frac{r^\sigma}{c}M_\sigma^+\left( c u\left(\frac{\,\cdot\,-x_0}{r}\right)- c v\left(\frac{\,\cdot\,-x_0}{r}\right)\right)(x_0 + rx)\\
& = M_\sigma^+(u-v)(x),
\end{split}\]
since $\I$ is elliptic with respect to $\mathcal L_0(\sigma)$ and $M_\sigma^+$ is translation invariant, positively homogeneous of degree one, and scale invariant of order $\sigma$.
Similarly,
\[ M_\sigma^-(u-v)(x)\le \bigl\{{\rm \tilde I}\, u-{\rm \tilde I}\, v\bigr\}(x).\]

\end{proof}

The following proposition immediately implies Theorem \ref{thm1}. However the statement of the proposition is better suited for a proof  by contradiction.
\begin{prop}\label{prop1}
Let $\sigma_0\in(0,2)$ and $\sigma\in{[\sigma_0,2]}$.
Let $u\in L^{\infty}\bigl(\R^n \times {(-1,0)}\bigr)$ be a viscosity solution of $u_t - \I u =f$ in $B_1\times(-1,0]$, where $\I$ is a translation
invariant elliptic operator with respect to the class $\mathcal L_0(\sigma)$.
Let $\alpha=\alpha(\sigma_0)$ be given by Theorem \ref{holder}.

Then, for all $\beta<\min\{\sigma_0,1+\alpha\}$, $u(\,\cdot\,,t)$
belongs to $C^\beta\left(\overline {B_{1/2}}\right)$ for all $t\in[-1/2,0]$, and $u(x,\,\cdot\,)$ belongs to $C^{\beta/{\sigma}}\left([-1/2,0]\right)$ for all $x\in B_{1/2}$. Moreover, the
following estimate holds
\[ \sup_{t\in[-1/2,0]} \bigl\|u(\,\cdot\,,t)\bigr\|_{C^\beta(B_{1/2})}+ \sup_{x\in B_{1/2}} \bigl\|u(x,\,\cdot\,)\bigr\|_{C^{\beta/{\sigma'}}\left([-1/2,0]\right)} \le C C_0\]
where
\[ C_0 =  \|u\|_{L^\infty\left(\R^n\times(-1,0)\right)} + \|f\|_{L^\infty\left(B_1\times(-1,0)\right)} \]
and  $C$ depends only on $\sigma_0$, $\beta$, ellipticity constants, and dimension.
\end{prop}

\begin{proof}
For $r\in(0,+\infty]$, we denote
\[ Q_r = B_r\times(-r,0].\]
Note that we may consider $u$ to be defined in the whole $Q_\infty$ and not only in $\R^n\times(-1,0]$ by
extending $u$ by zero. This is only by notational convenience and there is no difference in doing it since the equation is local in time and its domain will still be $Q_1$.

The proof is by contradiction. Suppose that the statement is false, i.e., there are sequences of functions $u_k$, $f_k$, operators $\I_k$, and orders $\sigma_k\in [\sigma_0,2]$ such that
\begin{itemize}
\item $\partial_t u_k- \I_k u_k=f_k$ in $B_1\times(-1,0]$
\item $\I_k$ is translation invariant and elliptic with respect to $\mathcal L_0(\sigma_k)$
\item $\|u_k\|_{L^\infty(Q_{\infty})}+ \|f_k\|_{L^\infty(Q_1)}=1$ (by scaling to make $C_0=1$)
\end{itemize}
but
\begin{equation} \label{1}
\sup_{t\in[-1/2,0]} \bigl\|u_k(\,\cdot\,,t)\bigr\|_{C^\beta(B_{1/2})}+ \sup_{x\in B_{1/2}} \bigl\|u_k(x,\,\cdot\,)\bigr\|_{C^{\beta/{\sigma_k}}\left([-1/2,0]\right)} \nearrow \infty.
\end{equation}
We split the proof in two cases: $\sigma_0 \le 1$ and $\sigma_0 >1$.

\vspace{3pt}
{\it Case $\sigma_0\le 1$.}
Since in this case we have $\beta<1$,   \eqref{1} is equivalent to
\begin{equation}
\sup_k \sup_{r>0} \sup_{z\in Q_{1/2}} r^{-\beta}\bigr\| u_k - u_k(z) \bigr\|_{L^\infty(Q_r^{\sigma_k}(z))} = \infty.
\end{equation}

Define the quantity
\[ \Theta(r) :=  \sup_k \sup_{r' \ge r}  \sup_ {z\in Q_{1/2}}  (r')^{-\beta} \bigl\|u_k-u_k(z)\bigr\|_{L^\infty\left(Q_{r'}^{\sigma_k}(z)\right)},\]
which is monotone in $r$.
Note that we have $\Theta(r)<\infty$ for $r>0$ and $\Theta(r)\nearrow \infty$ as $r\searrow0$. Clearly, there are sequences $r_m\searrow 0$, and $k_m$, and $z_{m}\in \overline{Q_{1/2}}$ for which
\begin{equation}\label{nondeg2}
(r_m)^{-\beta}\left\|{u_{k_m}-u_{k_m}(z_m)}\right\|_{L^\infty(Q_{r_m}^{\sigma_{k_m}}(z_{m}))} \ge \Theta(r_{m})/2.
\end{equation}

In this situation, let us denote $z_m= (x_m, t_m)$, $\sigma_m=\sigma_{k_m}$, and consider the blow up sequence
\[ v_m(x,t) = \frac{u_{k_m}(x_{m} +r_m x, t_m + (r_m)^{\sigma_m}t)-u_{k_m}(z_m)}{(r_m)^{\beta}\Theta(r_m)}.\]
Note that we will have, for all $m\ge 1$,
\begin{equation}\label{2}
v_m(0) =0  \quad \mbox{and}\quad  \|v_m\|_{L^\infty(Q_{1})}\ge 1/2.
\end{equation}
The last inequality is a consequence of \eqref{nondeg2}

For all $R\ge 1$,   $v_m$ satisfies the growth control
\begin{equation}\label{gctrl}
\begin{split}
\|v_{m}\|_{L^\infty(Q_R^{\sigma_m})} &= \frac{1}{(r_m)^{\beta} \Theta({r_m})} \bigl\|u_{k_m}- u_{k_m}(z_m)\bigr\|_{L^\infty\left(Q_{r_m R}^{\sigma_m}(z_{m})\right)}
\\
&= \frac{R^{\beta}}{\Theta({r_m}) (r_m R)^{\beta}}  \bigl\|u_{k_m}- u_{k_m}(z_m)\bigr\|_{L^\infty\left(Q_{r_m R}^{\sigma_m}(z_{m})\right)}
\\
&\le \frac{R^{\beta}\Theta({r_m}R)}{\Theta({r_m})}
\\
&\le   R^{\beta},
\end{split}
\end{equation}
where we have used the definition of $\Theta(r)$ and its montonicity.

For all fixed  $\rho \le (1-2^{-\sigma_0})/ r_m \nearrow \infty$, then $v_{m}$ solves
\begin{equation}\label{eqvm}
\bigl(\partial_t v_m - {\rm \tilde I}_m v_{m}\bigr)(x,t) = \frac{(r_m)^{\sigma_m}}{(r_m)^{\beta}\Theta(r_m)} f(x_m +r \,\cdot\,, t_m + r^{\sigma_m} t )\quad \mbox{in }B_{\rho} \times (-\rho^{\sigma_0}, 0]\,,
\end{equation}
where ${\rm \tilde I}_m$ is the operator ${\rm  I}_{k_m}$  appropriately rescaled.
More precisely,  given an elliptic test function $w:\R^n\to \R$ it is
\[ {\rm \tilde I}_m \left( \frac{w(x_m +r\,\cdot) - u_{k_m}(z_m)}{(r_m)^{\beta}\Theta(r_m)} \right) = \frac{(r_m)^{\sigma_m}}{(r_m)^{\beta}\Theta(r_m)} \bigl(\I_{k_m} w\bigr)(\,\cdot\,). \]
By the proof of Lemma \ref{lemscalings}, $\tilde \I_m$ is elliptic with respect to $\mathcal L_0(\sigma_m)$ with the same ellipticity constants.

Note that since $\beta<\sigma_0\le\sigma_m$ and $\Theta(r_m)\nearrow \infty$, the right hand sides of \eqref{eqvm} converge uniformly to $0$, and in particular they are uniformly bounded.
Then, using the $C^\alpha$ estimate in Theorem \ref{holder} (rescaled) in every cylinder $B_{\rho} \times (-\rho^{\sigma_0}, 0]$, $\rho>1$,  we obtain a subsequence $v_{m}$ converging locally uniformly
in all of $\R^n\times(-\infty,0]$ to some function $v$. Note that, although these $C^\alpha$ estimates for $v_m$  clearly depend on $\rho$, the important fact is that they are independent of $m$. This is enough  to
obtain local uniform convergence by the Arzel\`a-Ascoli Theorem and the typical diagonal argument. Moreover, since all the $v_{m}$'s satisfy the  growth control
\[ \|v_{m}\|_{L^\infty(Q_R^{\sigma_0})} \le  \|v_{m}\|_{L^\infty(Q_R^{\sigma_m})} \le R^\beta \]
and  $\beta< \sigma_0$, by dominated convergence we obtain that
\[\sup_{t\in [-\rho^{\sigma_0}, 0]} \int \bigl|v_m -v\bigr|(x,t) \omega_{\sigma_0}(x)\,dx \rightarrow 0.\]

In addition, by Lemma \ref{lemweak} there is a subsequence of ${\rm \tilde I}_m$
which converges weakly to some operator ${\rm \tilde I}$, translation invariant and elliptic with respect to $\mathcal L_0(\sigma)$ for some
$\sigma\in[\sigma_0,2]$ in every ball $B_R$. Hence, it follows from Lemma \ref{lem5parab} that $v$ satisfies
\[
v_t - {\rm \tilde I} v =0
\quad \mbox{in all of }\R^n\times(-\infty,0].
\]

On the other hand, by local uniform convergence, passing to the limit the growth controls \eqref{gctrl} for each $v_m$ we obtain that
$\|v\|_{L^\infty(Q_R^{\sigma})} \le   R^{\beta}$. Hence, by Theorem \ref{liouville}, $v$ must be constant. But passing \eqref{2} to the limit we obtain
$v(0)=0$ and $\|v\|_{L^\infty(Q_{1})}\ge 1/2$ and hence $v$ is not constant; a contradiction.

\vspace{3pt}
{\it Case $\sigma_0 >1$.}
In this case it is enough to consider $1<\beta<\min\{\sigma_0,1+\alpha\}$.
By Lemma \ref{lemaux}, \eqref{1} implies
\begin{equation}\label{111}
\sup_k \sup_{r>0} \sup_{z\in Q_{1/2}} r^{-\beta}\bigr\| u_k - \ell_{k,r,z} \bigr\|_{L^\infty(Q_r^{\sigma_k}(z))} = \infty,
\end{equation}
where, as in Lemma \ref{lemaux},  $\ell_{k,r,z}$ is the affine function of the variables $x$ only which best fits $u_k$ in $Q_r^{\sigma_k}(z)$ by least squares. Namely,
\[ \ell_{k,r,z}(x) = a^*(k,r,z) \cdot (x-z') +b^*(k,r,z)\]
for
\[ \bigl(a^*(k,r,z),b^*(k,r,z)\bigr) = {\rm arg\,min}_{(a,b)\in\R^n\times \R} \int_{Q_r^{\sigma_k}(z)} \bigl(u_k(x,t) - a\cdot (x- z') +b \bigr)^2\,dx\,dt, \]
where $z'$ denotes the first $n$ components of  $z$.

Now we define the quantity
\[
\Theta(r) :=  \sup_k \sup_{r'\ge r}  \sup_ {z\in Q_{1/2}}  (r')^{-\beta}  \bigl\|u_k-\ell_{k,r',z}\bigr\|_{L^\infty\left(Q_{r'}^{\sigma_k}(z)\right)},
\]
which is monotone in $r$.
Notice that we have $\Theta(r)<\infty$ for $r>0$ and $\Theta(r)\nearrow \infty$ as $r\searrow0$. Again, there are sequences $r_m\searrow 0$, and $k_m$, and $z_{m}\in \overline{Q_{1/2}}$ for which one (or more) of the following 
three possibilities happen
\begin{equation}\label{nondeg2b}
(r_m)^{-\beta}\left\|{u_{k_m}-\ell_{k_m,r_m,z_m}}\right\|_{L^\infty(Q_{r_m}^{\sigma_{k_m}}((z_{m}))} \ge \Theta(r_{m})/2
\end{equation}

We then denote $z_m= (x_m, t_m)$, $\sigma_m=\sigma_{k_m}$, $\ell_m = \ell_{k_m, r_m,z_m}$, and consider the blow up sequence
\[ v_m(x,t) = \left(\frac{u_{k_m}-\ell_m}{(r_m)^{\beta}\Theta(r_m)} \right)\bigr(x_{m} +r_m x, t_m +(r_m)^{\sigma_m} t\bigl).\]
Note that we will have, for all $m\ge 1$,
\begin{equation}\label{4}
\int_{Q_1} v_m \,dx\,dt =0 ,\quad  \int_{Q_1} v_m x_i \,dx\,dt = 0, \ 1\le i\le n,
\end{equation}
which are the optimality conditions of least squares. 

Translating  \eqref{nondeg2b} to $v_m$ we obtain that 
\begin{equation}\label{nondeg2c}
\left\|{v_m}\right\|_{L^\infty(Q_1)} \ge 1/2 
\end{equation}

Next we prove the growth control
\[
\|v_{m}\|_{L^\infty(Q_R^{\sigma_m})}  \le   CR^{\beta},\quad \mbox{for all }R\ge1.
\]
Indeed, for all $k$, $z\in Q_{1/2}$ and  $r' \ge r$ we have, by definition of $\Theta(r)$, 
By definition of $\Theta$, for all $z\in \overline{Q_{1/2}}$, $r>0$, and $\bar z  \in  Q_{r'}^{\sigma_k}(z)$ we have
\[\begin{split}
\bigl| \ell_{k,2r',z} (\bar z) - \ell_{k,r',z} (\bar z)\bigr|
&\le  \bigl| u(\bar z) -\ell_{k,2r',z} (\bar z)\bigr| + \bigl| u(\bar z) -\ell_{r,z} (\bar z)\bigr|\\
&\le C \Theta(r) (r')^{\beta}.
\end{split}\]

This happening for all $\bar z  \in  Q_{r'}^{\sigma_k}(z)$ implies
\[  \frac{r'\bigl|a^*(k,2r',z)-a^*(k,r',z)\bigr|} { (r')^\beta \Theta(r)}\le C \quad \mbox{ and }\quad \frac{ \bigl|b^*(k,2r',z)-b^*(k,r',z)\bigr|}{ (r')^\beta \Theta(r)}  \le C.\]
And thus, setting  $R= 2^N$, where $N\ge1$ is an integer, we have
\[ 
\begin{split}
 \frac{r\bigl|a^*(k,Rr,z)-a^*(k,r,z)\bigr|} { r^\beta \Theta(r)}& \le C \sum_{j=0}^{N-1}  2^{j(\beta-1)}\frac{2^jr\bigl|a^*(k,2^{j+1}r,z)-a^*(k,2^jr,z)\bigr|} { (2^jr)^\beta \Theta(r)}
\\&
\le C 2^{(\beta-1)N} = CR^{\beta-1}.
\end{split}\]
Similarly, 
\[ 
 \frac{\bigl|b^*(k,Rr,z)-b^*(k,r,z)\bigr|} { r^\beta \Theta(r)}\le CR^{\beta}.
\]

Therefore, for all $R\ge 1$
\begin{equation}\label{gctrl}
\begin{split}
\|v_{m}\|_{L^\infty(Q_R^{\sigma_m})} &= \frac{1}{(r_m)^{\beta} \Theta({r_m})} \bigl\|u_{k_m}-\ell_{k_m, r_m,z_m}\bigr\|_{L^\infty\left(Q_{Rr_m}^{\sigma_m}(z_{m})\right)}
\\
&\le  \frac{1}{(r_m)^{\beta} \Theta({r_m})} \bigl\|u_{k_m}-\ell_{k_m, Rr_m ,z_m}\bigr\|_{L^\infty\left(Q_{Rr_m }^{\sigma_m}(z_{m})\right)} + \\
& \hspace{10mm}+  \frac{1}{(r_m)^{\beta} \Theta({r_m})} \bigl\|\ell_{k_m, Rr_m ,z_m}- \ell_{k_m, r_m ,z_m}\bigr\|_{L^\infty\left(Q_{Rr_m}^{\sigma_m}(z_{m})\right)} 
\\
&\le \frac{R^{\beta}\Theta(R{r_m})}{\Theta({r_m})} + CR^{\beta}
\\
&\le  C R^{\beta},
\end{split}
\end{equation}

Next, for all fixed  $\rho \le (1-2^{-\sigma_0})/ r_m \nearrow \infty$, $v_{m}$ solves
\begin{equation}\label{eqvm2}
\bigl(\partial_t v_m - {\rm \tilde I}_m v_{m}\bigr)(x,t) = \frac{(r_m)^{\sigma_m}}{(r_m)^{\beta}\Theta(r_m)} f(x_m +r \,\cdot\,, t_m + (r_m)^{\sigma_m} t )\quad \mbox{in }B_{\rho} \times (-\rho^{\sigma_0}, 0]\,,
\end{equation}
where ${\rm \tilde I}_m$ is defined by
\[ {\rm \tilde I}_m \left( \frac{w(x_m +r_m\,\cdot) -  \ell_m(x_m +r_m\,\cdot)}{(r_m)^{\beta}\Theta(r_m)} \right) = \frac{(r_m)^{\sigma_m}}{(r_m)^{\beta}\Theta(r_m)} \bigl(\I_{k_m} w\bigr)(\,\cdot\,). \]
By Lemma \ref{lemscalings} $\tilde \I_m$ is elliptic with respect to $\mathcal L_0(\sigma_m)$,.

As a consequence, repeating the reasoning in the first part of the proof, a subsequence of $v_m$ converges locally uniformly in $\R^n\times(-\infty,0]$ to a function $v$ which satisfies
$v_t = {\rm \tilde I} v$ for some  ${\rm \tilde I}$ in translation invariant and elliptic  with respect to $\mathcal L_0(\sigma)$ with $\sigma\in[\sigma_0,2]$.
Hence, $v$ satisfies the limit growth control of the $v_m$'s, and thus  by Theorem \ref{liouville}, $v= a\cdot x +b$. But passing \eqref{4} and \eqref{nondeg2c} to the limit we reach a contradiction.
\end{proof}

We finally give the
\begin{proof}[Proof of Theorem \ref{thm1}]
Let $\delta=\epsilon/4$ and divide $[\sigma_0, 2]$ into  $N= \lceil (2-\sigma_0)/\delta \rceil$ intervals $[\sigma_j,\sigma_{j+1}]$ , $j=0,1,2,\dots, N$, where $\sigma_N = 2$ and $ 0\le \sigma_{j+1}-\sigma_j \le \delta$.
For each of the intervals $[\sigma_j,\sigma_{j+1}]$ we use Proposition \ref{prop1}, with $\sigma_0$ replaced by $\sigma_j$.
We obtain that  the estimate of the Proposition holds for  $\beta =  \min\{\sigma_j, 1+\alpha\}-\delta$ with a constant $C_j$ that depends only $\delta$, $\sigma_j$, ellipticity constants, and dimension.
In particular, given $\sigma \in [\sigma_ 0,2]$ the estimate of the Theorem holds for all $\beta \le \min\{\sigma, 1+\alpha\ -2\delta\}$ with constant $C=\max C_j$.
\end{proof}

\section*{Acknowledgements}
The author is indebted to Dennis Kriventsov for fruitful discussions and suggestions on this paper.
The author is  also indebted to Xavier Cabr\'e, Xavier Ros-Oton, and Luis Silvestre for their enriching comments on a previous
version of this manuscript.

\end{document}